\documentclass[12pt]{amsart}
\usepackage{amssymb}
\usepackage[margin=1in]{geometry}
\usepackage{enumerate}

\DeclareMathOperator{\Tr}{Tr}

\DeclareMathOperator{\diverge}{div}

\DeclareMathOperator{\curl}{curl}

\DeclareMathOperator{\Op}{Op}

\DeclareMathOperator*{\slim}{s-lim}

\renewcommand{\Re}{\hbox{{\rm Re}}\,}
\newcommand{\abs}[1]{\lvert#1\rvert}
\newcommand{\norm}[1]{\lVert#1\rVert}
\newcommand{\Norm}[1]{\left\lVert#1\right\rVert}
\newcommand{\jap}[1]{\langle#1\rangle}

\newcommand{\N}{{\mathbb N}}
\newcommand{\R}{{\mathbb R}}

\newcommand{\T}{{\mathbb T}}
\newcommand{\calS}{{\mathcal S}}

\newcommand{\dd}{{\mathrm{d}}}

\newcommand{\w}{{\omega}}

\newcommand{\e}{{\varepsilon}}

\newcommand{\sph}{{{\mathbb S}^{d-1}}}

\newcommand{\wt}{\widetilde}
\newcommand{\ups}{v}

\numberwithin{equation}{section}

\newtheorem{theorem}{Theorem}[section]
\newtheorem{lemma}[theorem]{Lemma}

\newtheorem{proposition}[theorem]{Proposition}

\theoremstyle{definition}

\theoremstyle{remark}

\newtheorem*{remark*}{\bf Remark}

\numberwithin{equation}{section}

\begin{document}

\title[Spectral density of the scattering matrix]{The spectral density of the scattering matrix 
of the magnetic Schr\"odinger operator for high energies}

\author{Daniel Bulger and Alexander Pushnitski}
\address{Department of Mathematics,
King's College London,
Strand, London, WC2R~2LS, U.K.}
\email{daniel.bulger@kcl.ac.uk}
\email{alexander.pushnitski@kcl.ac.uk}



\begin{abstract}
The scattering matrix of the 
Schr\"odinger operator with smooth short-range electric and magnetic potentials is considered.
The asymptotic density of the eigenvalues of this scattering matrix 
in the high energy regime is determined.
An explicit formula for this density is given. 
This formula involves only the magnetic vector-potential.
\end{abstract}

\maketitle
\section{Main result and discussion}
\label{sec.a}
\subsection{Main result}
\label{sec.a1}
Consider the Schr\"odinger operator
$$
H = (i\nabla + A)^{2} + V
\quad\text{in}\quad
L^2(\R^d), \quad d\geq2;
$$
here $V:\R^d\to\R$ is an electric potential and 
$A = (A_{1},\dots,A_{d}):\R^d\to\R^d$ is a magnetic vector-potential.
We assume that both $V$ and $A$ are infinitely differentiable and
satisfy the estimates
\begin{equation}
\label{eq.a1}
\vert\partial^{\alpha}V(x)\vert\leq C_{\alpha}\jap{x}^{-\rho-\vert\alpha\vert},
\qquad
\vert\partial^{\alpha}A(x)\vert\leq C_{\alpha}\jap{x}^{-\rho-\vert\alpha\vert},
\qquad
\rho>1,
\end{equation}
for all multi-indices $\alpha$; here $\jap{x}=(1+\abs{x}^2)^{1/2}$. 
Let $H_0=-\Delta$; we denote by $S(k)$ the scattering matrix associated 
with the pair $H,H_{0}$ at the energy $\lambda = k^{2}>0$. 
We recall the precise definition of the scattering matrix in Section~\ref{sec.a2}; 
here we only note that $S(k)$ is a unitary operator in $L^2(\sph)$ and 
the operator  $S(k) - I$ is compact (see e.g. \cite{Ya2,Ya4}).
Thus, the spectrum of $S(k)$ consists of eigenvalues on the unit circle
$\T$; all eigenvalues (apart from possibly 1) have finite multiplicities
and can accumulate only to 1. 
Our aim is to describe the asymptotic behaviour of these eigenvalues
as $k\to\infty$.

If $A\equiv0$, one has
\begin{equation}
\label{eq.a11a}
\norm{S(k) - I} = O(k^{-1}),\quad k\to\infty \qquad (A\equiv0),
\end{equation}
and so the eigenvalues of $S(k)$ are located on an arc near
1 of length $O(k^{-1})$. 
The large energy asymptotics of $S(k)$ in this case is given by the 
Born approximation, see e.g. \cite[Chapter~8]{Ya2}; 
this makes the analysis of $S(k)$ rather explicit.
In \cite{BP1}, using the Born approximation, we 
have determined the large energy asymptotic density of 
the spectrum of $S(k)$  for $A\equiv0$; we will say more about 
this in the next subsection. 

When $A\not\equiv0$, the 
situation is radically different: as $k\to\infty$, the norm 
$\norm{S(k)-I}$ does \emph{not} tend to zero and the 
Born approximation is no longer valid. 
The high energy asymptotic expansion of the scattering amplitude (= the integral 
kernel of $S(k)-I$) for the magnetic Schr\"odinger operator was obtained
(through a very difficult proof)  in \cite{Ya3}.
Our main result below is merely a spectral consequence of \cite{Ya3}.

We need some notation. 
For any $\w\in\sph$, let $\Lambda_{\w}\subset\R^{d}$ be
the hyperplane passing through the origin and orthogonal to $\w$:
$$
\Lambda_\w=\{x\in\R^d: \jap{x,\w}=0\}.
$$
We equip both $\sph$ and $\Lambda_{\w}$ with the standard 
$(d-1)$-dimensional Lebesgue measure. We set
\begin{equation}
\label{eq.a4}
M(\w,\xi) 
= 
\int_{-\infty}^{\infty}
\jap{A(t\w+\xi),\w}\dd t,\quad\w\in\sph,
\quad\xi\in\R^d.
\end{equation}
Our main result is as follows:
\begin{theorem}
\label{thm.a1}
Let $V$, $A$ satisfy \eqref{eq.a1}.
Then for any function $\varphi\in C(\mathbb{T})$ that vanishes 
in a neighbourhood of the point $1$,
\begin{equation}
\label{eq.a8}
\lim_{k\rightarrow\infty}k^{-d+1}\Tr\varphi(S(k)) 
= 
(2\pi)^{-d+1}
\int_\sph
\int_{\Lambda_\w}\varphi(e^{iM(\w,\xi)})\dd \xi \dd\w.
\end{equation}
\end{theorem}
Under the assumptions of Theorem~\ref{thm.a1}, 
the operator $\varphi(S(k))$ has a finite rank, 
and so the trace in the l.h.s. exists; also, the integrand
in the r.h.s. vanishes for all sufficiently large $\abs{\xi}$, and so 
the integral is absolutely convergent. 

The striking feature of formula \eqref{eq.a8} is that 
its r.h.s.\  \emph{does not depend on the electric potential $V$}.

\subsection{Discussion}

\emph{Weak convergence of measures:}
Theorem~\ref{thm.a1} can be rephrased in terms of weak convergence
of the eigenvalue counting measures. 
Let us denote the eigenvalues of $S(k)$ by $e^{i\theta_n(k)}$, where 
$n\in\N$ and 
$\theta_{n}(k)\in [-\pi,\pi)$; as usual, the eigenvalues are counted
with multiplicities taken into account. 
Let $\iota\subset \T\setminus\{1\}$ be any open arc separated away from $1$; 
we set 
\begin{align}
\label{eq.muk}
\mu_{k}(\iota) 
&= 
\#\{n\in\N : e^{i\theta_n(k)}\in\iota\}, \quad k>0,
\\
\mu(\iota)
&=
(2\pi)^{-d+1} 
\int_\sph
\int_{\Pi_\w(\iota)} \dd \xi \dd\w, 
\qquad
\Pi_\w(\iota)=\{\xi\in\Lambda_\w: e^{iM(\w,\xi)}\in\iota\},
\label{eq.a6}
\end{align}
where 
$\#$ represents the number of elements in a set.
The measures $\mu_k$, $\mu$ can be alternatively defined by 
requiring that
\begin{align}
\Tr\varphi(S(k))
&=
\int_\T \varphi(z)\dd\mu_k(z), 
\label{eq.a7a}
\\
(2\pi)^{-d+1}
\int_\sph
\int_{\Lambda_\w}\varphi(e^{iM(\w,\xi)})\dd \xi \dd\w
&=
\int_{\mathbb{T}}\varphi(z)\dd\mu(z),
\label{eq.a7b}
\end{align}
for any continuous function $\varphi$ on $\T$ vanishing in a 
neighbourhood of $1$. 
With this notation,  Theorem~\ref{thm.a1} can be rephrased as  the weak convergence
$$
k^{-d+1}\mu_k\to\mu, \quad k\to\infty. 
$$
The measure $\mu$ may be singular at $1$, i.e.  
$\mu(\T\setminus\{1\})$ may be infinite, but
\begin{equation}
\label{eq.a7}
\int_{\mathbb{T}}\vert z-1\vert^{\ell}\dd\mu(z)<+\infty,\quad\forall\ell>(d-1)/(\rho-1).
\end{equation}
The relation \eqref{eq.a7} follows from the elementary 
estimate on $M(\w,\xi)$ which is  a direct consequence of \eqref{eq.a1}:
\begin{equation}
\label{eq.a5}
\vert M(\w,\xi)\vert \leq C(A)(1+\vert\xi\vert)^{1-\rho},\quad\w\in\sph,\quad\xi\in\Lambda_{\w}.
\end{equation}

\emph{The case $A\equiv0$:}
In this case, the limiting measure $\mu$ vanishes, and in order to obtain
non-trivial asymptotics of $\mu_k$, the problem requires an appropriate 
rescaling. 
By \eqref{eq.a11a}, the spectrum of $S(k)$ consists of 
eigenvalues which lie on an arc of length $O(k^{-1})$ around $1$. 
This suggests the following rescaled version of the problem: 
for an interval $\delta\subset\R\setminus\{0\}$ separated away from zero, 
set 
$$
\wt \mu_k(\delta)
=
\#\{n\in\N : k\theta_{n}(k)\in\delta\}.
$$
Then it turns out (see \cite{BP1}) that 
\begin{equation}
k^{-d+1}\wt \mu_k\to\wt\mu
\quad\text{weakly as} \quad
k\to\infty,
\label{eq.a11b}
\end{equation}
where the measure $\wt\mu$ is defined as follows. 
Let
\begin{equation}
X(\w,\xi) = -\frac{1}{2}\int_{-\infty}^{\infty}V(t\w + \xi)\dd t,\quad\w\in\sph,\quad\xi\in\Lambda_{\w};
\label{eq.a11c}
\end{equation}
then
\begin{equation*}
\wt{\mu}(\delta) 
= 
(2\pi)^{-d+1}
\int_\sph
\int_{\wt\Pi_\w(\iota)} \dd \xi \dd\w, 
\qquad
\wt\Pi_\w(\iota)=\{\xi\in\Lambda_\w: X(\w,\xi)\in\delta\},
\end{equation*}
where $\delta\subset\R\backslash\{ 0\}$.

\emph{Two and three dimensional cases:}
Let us discuss assumptions \eqref{eq.a1} in dimensions $d=2,3$. 
In dimension $d=3$, a magnetic
vector-potential $A$ satisfying \eqref{eq.a1} can be constructed for any 
smooth magnetic field $B:\R^3\to\R^3$ ($B=\curl A$) such that 
$\diverge B=0$ and 
\begin{equation}
\abs{\partial^\alpha B(x)}\leq C_\alpha \jap{x}^{-\rho-1-\abs{\alpha}}, 
\quad 
\rho>1. 
\label{eq.a2a}
\end{equation}
In dimension $d=2$, the magnetic field $B:\R^2\to\R$, in addition to 
\eqref{eq.a2a}, must satisfy the zero flux condition 
\begin{equation}
\Phi=\int_{\R^2} B(x)\dd x=0, 
\quad
B(x)=\frac{\partial A_1}{\partial x_2}-\frac{\partial A_2}{\partial x_1}. 
\label{eq.a2b}
\end{equation}
See \cite{Ya4} for the details of this construction.
The zero flux condition is unavoidable in the following sense. 
Let $B:\R^2\to\R$ be given which satisfies the estimates \eqref{eq.a2a} but
the flux $\Phi\not=0$. 
Then any magnetic vector-potential $A(x)$ for this field will necessarily fail to be
short-range (i.e. \eqref{eq.a1} fails) but one can construct $A(x)$ with the behaviour
$\abs{A(x)}\sim \abs{x}^{-1}$ as $\abs{x}\to\infty$. 
In this case the scattering theory for $H_0$, $H$ can still be constructed, but the 
difference $S(k)-I$ will \emph{not} be compact, see \cite{Ya4} for a detailed 
discussion and a description of the essential spectrum of $S(k)$. 
A particularly well known example of this is the Aharonov-Bohm effect \cite{AhBo}. 
Thus, in this case the measure $\mu_k$ cannot even be defined and 
the question of the spectral asymptotics of the scattering matrix 
cannot be approached in the same way as in Theorem~\ref{thm.a1}.

\emph{Gauge invariance:}
The scattering matrix is gauge invariant in the class of short-range 
magnetic vector-potentials. More precisely, if 
\begin{equation}
\wt{A}(x) = A(x) + \nabla\phi(x), 
\quad\text{where}\quad
\phi(x)\to0
\quad\text{as}\quad
\abs{x}\to\infty,
\label{eq.a3a}
\end{equation}
then the scattering matrix $\wt S(k)$ associated with the pair 
$\wt H=(i\nabla+\wt A)^2+V$, $H_0=-\Delta$, 
coincides with $S(k)$. See \cite{Ru,Ya4} for further details
and for an interesting discussion of examples when $S(k)$ is not 
gauge invariant (when $\phi(x)$ is homogeneous in $\abs{x}$ of order zero). 
An inspection shows that $M(\w,\xi)$ is gauge invariant under the gauge
transformations of the class \eqref{eq.a3a}.

\emph{Related work:}
Although the study of various aspects of asymptotic distribution of eigenvalues
of differential operators has become a well-developed industry, very little
attention in mathematical literature has been devoted to analogous questions 
for the scattering matrix.  
There has been, of course, much work on the ``average'' characteristics
of the scattering matrix $S(k)$: on the scattering phase
$\arg\det S(k)$ and on the total scattering cross section; but the 
study of the distribution of individual eigenvalues of $S(k)$ has been less popular.
We are only aware of mathematical works \cite{Ko,BYa1,BYa2,ZZ} on this subject. 
In \cite{Ko}, the asymptotic behaviour of the phases $\theta_n(k)$
was determined for a fixed $k$ and $n\to\infty$ for potentials $V$ with 
compact support and $A\equiv0$. 
In \cite{BYa1}, the same problem was considered in the case of potentials 
$V$ with a power asymptotics at infinity. In \cite{BYa2}, this question was
studied in the presence of a periodic background potential.
In \cite{ZZ}, the spacing between the phases $\theta_n(k)$ was studied
in a rather special model of scattering on manifolds.

In the physics community, the question of asymptotic distribution of 
eigenvalues of the scattering matrix has been addressed; see e.g.
the works by  U.~Smilansky and his collaborators \cite{Smi} 
on the eigenvalue statistics of $S(k)$ for obstacle scattering. 

The limiting measure $\mu$ arises via integration over straight lines,
i.e. over the trajectories of the free dynamics, see \eqref{eq.a4}.
Similar in nature asymptotic formulas are known in other 
problems, involving the discrete spectrum of differential and 
pseudodifferential operators; see \cite{wein,tvb,urvb,PRVB}.

\subsection{Key steps of the proof of Theorem~\ref{thm.a1}}
\label{sec.a6}
Our proof is heavily based on the results of \cite{Ya3}.
In \cite{Ya3}, D.~Yafaev suggested 
a high energy asymptotic expansion for the integral kernel of the scattering 
matrix $S(k)$.
This expansion is constructed via approximate scattering 
solutions to the stationary Schr\"odinger equation
\begin{equation*}
Hu = k^{2}u.
\end{equation*}
We recall this construction in Section~\ref{sec.b}. 
Using Yafaev's expansion, it is easy to represent
the scattering matrix $S(k)$ (modulo a negligible error) 
as a semiclassical pseudodifferential operator ($\Psi$DO)
on the sphere with the semiclassical parameter $h=k^{-1}$ 
and the principal symbol $e^{iM(\w,\xi)}$. 
This representation is already present, in a somewhat different 
form, in \cite{Ya3}; we re-derive it in Section~\ref{sec.b}
in the form convenient for our purposes (see Lemma~\ref{prop.b6}). 
After this, using the standard pseudodifferential techniques, we
prove (see Lemma~\ref{prop.c1}) the asymptotic formula
\begin{equation}
\label{eq.c1}
\lim_{k\to\infty}k^{-d+1}
\Tr(S(k)-I)^{\ell_1}(S(k)^*-I)^{\ell_2}
= 
(2\pi)^{-d+1}
\int_\sph \int_{\Lambda_\omega}
(e^{iM(\omega,\xi)}-I)^{\ell_1}(e^{-iM(\omega,\xi)}-I)^{\ell_2}\dd \xi \dd\omega
\end{equation}
for any integers $\ell_1\geq0$, $\ell_2\geq0$ such that 
the sum $\ell_1+\ell_2$ is sufficiently large. 
From here it is easy to derive the main result by an application of the 
Weierstrass approximation theorem; this is done in Section~\ref{sec.c}.

The difference between the case $A\equiv0$ and the general case can be understood as 
follows. As mentioned above, the scattering matrix can be approximated by a $\Psi$DO 
on the sphere with a symbol which depends on $k$. 
The leading term of the asymptotic expansion of this symbol in powers of $k^{-1}$ 
is $e^{iM(\omega,\xi)}$; this term involves only the magnetic vector potential $A$.
The electric potential turns up only in the next term of the expansion, which 
(if $A\equiv0$) has the form $ik^{-1}X(\omega,\xi)$ (see \eqref{eq.a11c}). 
Of course, this is related to the fact that the magnetic potential
is a perturbation of the order 1 (as a differential operator) of $-\Delta$, whereas
the electric field is a perturbation of the order 0.

\section{The scattering matrix as a \protect$\Psi$DO on the sphere.}
\label{sec.b}

\subsection{The scattering matrix}
\label{sec.a2}
Let us briefly recall the definition of the scattering matrix; for the details, 
we refer to any textbook on scattering theory, e.g. \cite{Ya2}.
It is one of the fundamental facts of scattering theory that under the conditions 
\eqref{eq.a1}, the wave operators
\begin{equation*}
W_{\pm} = \slim_{t\to\pm\infty}e^{itH}e^{-itH_{0}}
\end{equation*}
exist and are complete.
The scattering operator $\mathbf{S} = W_{+}^{*}W_{-}$ is 
unitary in $L^{2}(\R^{d})$ and commutes with $H_{0}$. Let 
$F: L^2(\R^d)\to L^2((0,\infty);L^2(\sph))$ be the unitary operator
$$
(Fu)(k,\omega)=\frac1{\sqrt{2}}k^{(d-2)/2}\widehat u(k\omega), 
\quad k>0,\quad \omega\in\sph,
$$
where $\widehat u$ is the usual (unitary) Fourier transform of $u$. 
The operator $F$ diagonalises $H_0$, i.e.
$$
(FH_0 u)(k,\omega)
=
k^{2} (Fu)(k,\omega),
\quad
\forall u\in C_0^\infty(\R^{d}).
$$
The commutation relation $\mathbf SH_0=H_0\mathbf S$ 
implies that $F$ also diagonalises $\mathbf S$, i.e.
$$
(F\mathbf Su)(k,\cdot)=S(k)(Fu)(k,\cdot),
$$
where $S(k):L^2(\sph)\to L^2(\sph)$ 
is the unitary operator known as the (on-shell) scattering matrix.

\subsection{Pseudodifferential operators on the sphere}
\label{sec.b1}
For every $\w\in\sph$, we identify the cotangent space $T_{\w}^{*}\sph$ 
with the plane $\Lambda_{\w} = \{x\in\R^{d} : \jap{x,\w} = 0\}$
in a standard way.
For a symbol $\sigma\in C_{0}^{\infty}(T^{*}\sph)$ and a semiclassical parameter 
$h\in(0,1)$, the semiclassical $\Psi$DO $\Op_{h}[\sigma]$ in $L^{2}(\sph)$ is defined via its integral kernel
\begin{equation}
\label{eq.b1}
\Op_{h}[\sigma](\w,\w') = (2\pi h)^{-d+1}\int_{\Lambda_{\w}}e^{-i\jap{\w-\w',\xi}/h}\sigma(\w,\xi)\dd\xi,
\end{equation}
where $\w,\w'\in\sph$. 
This definition can be extended in a standard way to symbols $\sigma$ satisfying
\begin{equation}
\label{eq.b2}
\vert\partial_{\xi}^{\alpha}\partial_{\w}^{\beta}\sigma(\w,\xi)\vert\leq C_{\alpha\beta}\jap{\xi}^{-m-\vert\alpha\vert},\quad\w\in\sph,\quad\xi\in\Lambda_{\w},
\end{equation}
for some $m\in\R$ and all multi-indices $\alpha,\beta$. 
We will only be interested in the case $m>0$; then 
by the Calderon-Villancourt theorem (see e.g. \cite{Taylor})
combined with a scaling argument, $\Op_{h}[\sigma]$ is bounded and 
\begin{equation}
\label{eq.b3b}
\sup_{0<h<1}\norm{\Op_{h}[\sigma]} 
\leq C(\sigma).
\end{equation}
Next, if $\sigma$ satisfies \eqref{eq.b2} with $m>d-1$, then 
(see e.g. \cite{SH1,DSj})
$\Op_{h}[\sigma]$ is trace class and its trace can be computed by integrating 
the kernel \eqref{eq.b1} over the diagonal:
\begin{equation}
\label{eq.b3}
\Tr(\Op_{h}[\sigma]) = (2\pi h)^{-d+1}\int_{\sph}\int_{\Lambda_{\w}}\sigma(\w,\xi)\dd\xi\dd\w.
\end{equation}

We will also be interested in symbols which depend on $h$. For $m> 0$, 
let $\calS(\jap{\xi}^{-m})$ be the class of $C^{\infty}$-smooth symbols 
$\sigma = \sigma(\w,\xi,h)$, $h\in(0,1)$, satisfying
the estimate \eqref{eq.b2} uniformly in $h\in(0,1)$ for all multi-indices $\alpha,\beta$.
We will need a standard statement about the leading term spectral asymptotics of a 
semiclassical $\Psi$DO:
\begin{proposition}
\label{prop.b2}
Let $m>0$ and let $\sigma\in {\mathcal S}(\jap{\xi}^{-m})$ be a symbol which admits the representation
\begin{equation}
\label{eq.b5}
\sigma = \sigma_{0} + h\sigma_{1}
\end{equation}
with $\sigma_{0},\sigma_{1}\in {\mathcal S}(\jap{\xi}^{-m})$ and $\sigma_{0}$ independent of $h$. 
Then for any non-negative integers $\ell_1$, $\ell_2$ such that 
$\ell_1+\ell_2>\frac{d-1}{m}$, the operator 
$(\Op_{h}[\sigma])^{\ell_1}\bigl((\Op_{h}[\sigma])^*\bigr)^{\ell_2}$ 
belongs to the trace class and 
\begin{equation}
\label{eq.b6}
\lim_{h\to+0}(2\pi h)^{d-1}
\Tr\bigl((\Op_{h}[\sigma])^{\ell_1}\bigl((\Op_{h}[\sigma])^*\bigr)^{\ell_2}\bigr)
= 
\int_{\sph}\int_{\Lambda_{\w}}
\sigma_{0}(\w,\xi)^{\ell_1}\overline{\sigma_{0}(\w,\xi)^{\ell_2}}
\dd\xi\dd\w.
\end{equation}
\end{proposition}
\begin{proof}[Sketch of proof] 
First suppose that $m>d-1$ and $\ell_1= 1$, $\ell_2=0$.  
Then $\Op_{h}[\sigma]$ is trace class
and the asymptotics \eqref{eq.b6} follow by substituting the representation 
\eqref{eq.b5} into \eqref{eq.b3}. 
In the general case, using local coordinates on the sphere, iterating a standard composition 
formula for $\Psi$DO in $L^{2}(\R^{d-1})$, and using the formula for the symbol of the adjoint
operator, we obtain that for any $N>0$, 
\begin{equation*}
(\Op_{h}[\sigma])^{\ell_1}\bigl((\Op_{h}[\sigma])^*\bigr)^{\ell_2}
= 
\Op_{h}[\varkappa] + R_N(h),
\end{equation*}
where:
\begin{enumerate}[(i)]
\item
the symbol $\varkappa\in {\mathcal S}(\jap{\xi}^{-m\ell_1-m\ell_2})$ can be represented as 
\begin{equation*}
\varkappa = \varkappa_{0} + h\varkappa_{1}
\end{equation*}
with $\varkappa_{0},\varkappa_{1}\in {\mathcal S}(\jap{\xi}^{-m\ell_1-m\ell_2})$, $\varkappa_{0}$ is independent of $h$ and
\begin{equation*}
\varkappa_{0}(\w,\xi) = \sigma_{0}(\w,\xi)^{\ell_1}\overline{\sigma_{0}(\w,\xi)^{\ell_2}};
\end{equation*}
\item
$R_N(h)$ is an operator with the integral kernel in 
$C^{N}(\sph\times\sph)$ and the $C^{N}$-norm of $R_N(h)$ is $O(h^{N})$ as $h\to 0$. 
\end{enumerate}

This reduces the problem to the case $\ell_1= 1$, $\ell_2=0$.
\end{proof}
\begin{remark*}
\label{rem.b2}
In the situation we are interested in, the representation \eqref{eq.b5} 
arises as a corollary of the asymptotic expansion
\begin{equation*}
\sigma \sim \sum_{j=0}^{\infty}h^{j}\sigma_{j},\quad\sigma_{j}\in {\mathcal S}(\jap{\xi}^{-m}),
\end{equation*}
but we are only interested in the first term of this expansion.
\end{remark*}
In our construction, the $\Psi$DO will be defined in terms of their amplitudes rather than their symbols. 
Thus, we need a statement which is standard in the 
$\Psi$DO theory (see e.g. \cite{SH1}).
\begin{proposition}
\label{prop.b3}
Let $m>0$, and let $b = b(\w,\w',\xi,h)$ be a smooth function of 
the variables
$(\w,\xi)\in T^{*}\sph$, $\w'\in\sph$ and $h\in (0,1)$.
Assume that $b$ satisfies the estimates
\begin{equation}
\label{eq.b7}
\vert\partial_{\xi}^{\alpha}\partial_{\w}^{\beta}\partial_{\w'}^{\gamma}b(\w,\w',\xi,h)\vert\leq C_{\alpha\beta\gamma}\jap{\xi}^{-m-\vert\alpha\vert}
\end{equation}
for all multi-indices $\alpha,\beta,\gamma$ uniformly in $h\in(0,1)$ . Then for any $N>0$, the operator with the integral kernel
\begin{equation}
\label{eq.b8}
(2\pi h)^{-d+1}\int_{\Lambda_{\w}}e^{-i\jap{\w-\w',\xi}/h}b(\w,\w',\xi,h)\dd\xi
\end{equation}
can be represented as $\Op_{h}[\sigma] + R_N(h)$, 
where the following conditions are met:
\begin{enumerate}[\rm (i)]
\item 
The symbol $\sigma$ can be written as $\sigma = \sigma_{0} + h\sigma_{1}$ 
with $\sigma_{0},\sigma_{1}\in {\mathcal S}(\jap{\xi}^{-m})$ and 
\begin{equation*}
\sigma_{0}(\w,\xi,h) = b(\w,\w,\xi,h).
\end{equation*}
\item 
The operator $R_N(h)$ has the integral kernel in $C^{N}(\sph\times\sph)$ 
with $C^{N}$-norm satisfying $O(h^{N})$ as $h\to 0$.
\end{enumerate}
\end{proposition}
\subsection{Approximate solutions to the Schr\"odinger equation}
\label{sec.b2}
Here we recall the construction of approximate solutions to the 
Schr\"odinger equation $Hu = k^{2}u$ from \cite{Ya3}. 
The solutions $u$ are sought as functions 
\begin{equation*}
u=u(x,p),\quad x\in\R^{d},\quad p\in\R^{d},\quad\vert p\vert = k.
\end{equation*}
We denote $\hat{p} = p\vert p\vert^{-1}\in\sph$. Writing 
\begin{equation*}
u(x,p) = e^{i\Theta(x,p)}\ups(x,p),
\end{equation*}
we obtain the eikonal equation 
\begin{equation*}
\vert\nabla\Theta\vert^{2} - 2\jap{A,\nabla\Theta}+(V(x)+\vert A(x)\vert^{2}) =k^{2}
\end{equation*}
for the phase function $\Theta$ and the transport equation
\begin{equation*}
-2i\jap{\nabla\Theta,\nabla \ups} + 2i\jap{A,\nabla \ups} 
- \Delta \ups + (-i\Delta\Theta + i\diverge A)\ups = 0
\end{equation*}
for the amplitude function $\ups$. The approximate solution to the eikonal equation is constructed as
\begin{equation}
\label{eq.b9}
\Theta_{\pm}(x,p) = \jap{x,p} + \phi_{\pm}(x,\hat{p}),
\end{equation}
\begin{equation}
\label{eq.b10}
\phi_{\pm}(x,\hat{p}) = \mp\int_{0}^{\infty}\jap{A(x\pm t\hat{p}),\hat{p}}\dd t.
\end{equation}
Next, for a given $N\in\N$, the approximate solution $\ups_{\pm}^{(N)}$ to the transport equation is constructed as the asymptotic series
\begin{equation}
\label{eq.b11}
\ups_{\pm}^{(N)}(x,p) = \sum_{n=0}^{N}(2ik)^{-n}\ups_{n}^{(\pm)}(x,\hat{p}),
\end{equation}
where $\ups_{0}^{(\pm)}(x,\hat{p}) \equiv 1$ and the coefficients 
$\ups_{n}^{(\pm)}$ are determined via an explicit iterative procedure:
$$
v_{n+1}^{(\pm)}(x,\hat{p})
=
\mp\int_0^\infty f_n^{(\pm)}(x\pm t\hat{p},\hat{p})\dd t,
$$
where
$$
f_n^{(\pm)}
=
2i\jap{A-\nabla\phi_\pm,\nabla v_n^{(\pm)}}
-\Delta v_n^{(\pm)}
+(\abs{\nabla\phi_\pm}^2-2\jap{A,\nabla\phi_\pm}+V_1-i\Delta \phi_\pm)v_n^{(\pm)},
$$
and 
$$
V_1=V+\abs{A}^2+i\diverge A.
$$
The functions $\phi_{\pm}$ and $\ups_{\pm}^{(N)}$ solve the eikonal and transport
equations up to error terms that can be explicitly controlled; we do not need the 
precise statement here, see \cite{Ya3} for the details. 
Here we need $\phi_{\pm}$ and $\ups_{\pm}^{(N)}$ simply as 
``building blocks'' for the approximation to the scattering amplitude, which is 
given in the next subsection. 

When considering the functions $\phi_{\pm}$ and $\ups_{n}^{(\pm)}$, 
we will always exclude a conical neighbourhood of the direction 
$\hat{x} = -\hat{p}$ (for the sign ``$+$'') or $\hat{x} = \hat{p}$ (for the sign ``$-$''). 
Outside these neighbourhoods, the functions $\phi_{\pm}$ and $\ups_{n}^{(\pm)}$, 
$n\geq 1$, decay at infinity in the $x$-variable.
More precisely, the following statement is proven in \cite{Ya3}:
\begin{proposition}\cite{Ya3}
\label{prop.b4}
Let assumption \eqref{eq.a1} hold and let $x\in\R^{d}$, 
$\w\in\sph$ be such that $\pm\jap{\hat{x},\w}\geq -1+\varepsilon$ 
for some $\varepsilon>0$. Then the functions $\phi_{\pm}$ 
and $\ups_{n}^{(\pm)}$, $n\geq 1$, satisfy the estimates
\begin{equation}
\label{eq.b12}
\vert\partial_{x}^{\alpha}\partial_{\w}^{\beta}\phi_{\pm}(x,\w)\vert
\leq 
C_{\alpha\beta}\jap{x}^{1-\rho-\vert\alpha\vert},
\end{equation}
\begin{equation}
\label{eq.b13}
\vert\partial_{x}^{\alpha}\partial_{\w}^{\beta}\ups_{n}^{(\pm)}(x,\w)\vert
\leq 
C_{\alpha\beta}\jap{x}^{-n-\vert\alpha\vert},
\end{equation}
for all multi-indices $\alpha,\beta$. 
\end{proposition}
We will write
\begin{equation}
\label{eq.b14}
u_{\pm}^{(N)}(x,p) = u_{\pm}(x,p) = e^{i\Theta_{\pm}(x,p)}\ups_{\pm}^{(N)}(x,p).
\end{equation}
\subsection{Approximation to the scattering amplitude}
\label{sec.b3}
Here we recall the approximation to the scattering amplitude obtained in \cite{Ya3}. 
It is known (see \cite{Agmon}) that off the diagonal $\w=\w'$, the integral kernel $s(\w,\w';k)$ 
of the scattering matrix $S(k)$ is a $C^{\infty}$-smooth function of $\w,\w'\in\sph$ and it tends to zero 
faster than any power of $k^{-1}$ as $k\to\infty$. 
Thus, it suffices to describe the structure of $s(\w,\w';k)$ in a neighbourhood of the diagonal $\w=\w'$. 
Fix some $\delta\in(0,1)$; 
for an arbitrary point  $\w_{0}\in\sph$, let 
$\Omega(\w_{0})\subset\sph$ be the conical neighbourhood of $\w_{0}$ given by
\begin{equation}
\Omega(\w_{0}) = \{\w\in\sph:\jap{\w,\w_{0}}>\delta\}.
\label{eq.b14a}
\end{equation}
Let $u_{\pm}$ be as in \eqref{eq.b14}. We set
$$
\partial_{\omega_0}u=\jap{\nabla u,\omega_0},
$$
where the gradient of $u=u(x,p)$ is taken in the $x$ variable. 
For $\w,\w'\in\Omega(\omega_0)$, define
\begin{multline}
s_{0}^{(N)}(\w,\w';k) = -i\pi k^{d-2}(2\pi)^{-d}\times 
\\
\times\bigg(\int_{\Lambda_{\w_{0}}}
\left[\overline{u_{+}^{(N)}(x,k\w)}(\partial_{\omega_0}u_{-}^{(N)})(x,k\w') 
- 
\overline{(\partial_{\omega_0}u_{+}^{(N)})(x,k\w)}u_{-}^{(N)}(x,k\w')\right]\dd x - 
\\
-2i\int_{\Lambda_{\w_{0}}}
\jap{A(x),\w_{0}}\overline{u_{+}^{(N)}(x,k\w)}u_{-}^{(N)}(x,k\w')\dd x\bigg).
\label{eq.b15}
\end{multline}
The integrals in \eqref{eq.b15} do not converge absolutely and should be understood as oscillatory integrals. In other words, \eqref{eq.b15} should be understood as a distribution on $\Omega(\omega_0)\times\Omega(\omega_0)$. 
\begin{proposition}\cite{Ya3}
\label{prop.b5}
For any $q\in\N$ there exists $N=N(q)\in\N$ such that for any $\w_{0}\in\sph$, the kernel
\begin{equation*}
\wt{s}^{(N)}(\w,\w';k) = s(\w,\w';k) - s_{0}^{(N)}(\w,\w';k)
\end{equation*}
belongs to the class $C^{q}(\Omega(\omega_0)\times\Omega(\omega_0))$, 
and its $C^{q}$-norm is $O(k^{-q})$ as $k\to\infty$.
\end{proposition}
\subsection{The scattering matrix as a \protect$\Psi$DO on the sphere}
\label{sec.b4}
Below we represent the scattering matrix $S(k)$ as a semiclassical $\Psi$DO on the sphere. 
The semiclassical parameter is $h=k^{-1}$. 
The idea to represent the scattering matrix as a $\Psi$DO on the sphere goes back
to \cite{BYa1}. The statements almost identical to Lemma~\ref{prop.b6} 
can be found in \cite[Propositions~6.1 and 6.4]{Ya3} and \cite[Section~8.4]{Ya2},
but for technical reasons these statements are not sufficient for our purposes.
A related work (but written from a very different viewpoint) is \cite{Al}, where
the scattering matrix is represented as a Fourier integral operator corresponding
to the classical scattering relation.

\begin{lemma}
\label{prop.b6}
Let assumptions \eqref{eq.a1} hold, and let $m=\min\{1,\rho-1\}$. 
For any $q\in\N$, 
the scattering matrix can be written as
\begin{equation}
\label{eq.b16}
S(k) = I + \Op_{k^{-1}}[\sigma] + R_q(k),
\end{equation}
where:
\begin{enumerate}[\rm (i)]
\item
The symbol $\sigma$
can be represented as 
\begin{equation}
\label{eq.b17}
\sigma = \sigma_{0} + k^{-1}\sigma_{1},
\end{equation}
with $\sigma_{0},\sigma_{1}\in {\mathcal S}(\jap{\xi}^{-m})$ and 
\begin{equation}
\label{eq.b18}
\sigma_{0}(\w,\xi) = \exp(iM(\w,\xi)) - 1;
\end{equation}
\item
the operator $R_q(k)$ has an integral kernel in the class 
$C^{q}(\sph\times\sph)$ and its $C^{q}$-norm is $O(k^{-q})$ as $k\to\infty$.\\
\end{enumerate}
\end{lemma}

\begin{proof}

1) Let $\psi_1,\psi_2\in C^{\infty}(\sph)$ be functions with disjoint supports. 
Then $\psi_1 S(k)\psi_2$ has a $C^{\infty}$-smooth integral kernel which decays faster 
than any power of $k^{-1}$ as $k\to\infty$. 
The same comment applies to $\psi_1\Op_{k^{-1}}[a]\psi_2$ with $a\in {\mathcal S}(\jap{x}^{-m})$. 
This shows that using a sufficiently fine partition of unity on the sphere, 
one easily reduces the problem to approximating the integral kernel 
of $S(k)$ locally in any conical neighbourhood $\Omega(\omega_0)$, see \eqref{eq.b14a}.
Thus, we can use Proposition~\ref{prop.b5}.

2) Let us rearrange the integrand in \eqref{eq.b15}. Denote
\begin{equation}
\label{eq.b19}
w_{\pm}(x,p) = e^{i\phi_{\pm}(x,\hat{p})}\ups_{\pm}^{(N)}(x,p),
\end{equation}
\begin{equation}
\label{eq.b20}
\wt{w}_{\pm}(x,p) = ke^{i\phi_{\pm}(x,\hat{p})}(\ups_{\pm}^{(N)}(x,p)-1),
\end{equation}
so that from \eqref{eq.b14},
\begin{equation*}
u_{\pm}(x,k\omega) 
= 
e^{ik\jap{x,\omega}}w_{\pm}(x,k\omega)
=
e^{ik\jap{x,\omega}}(e^{i\phi_\pm(x,\omega)}+k^{-1} \wt w_\pm(x,k\omega)),
\end{equation*}
\begin{equation*}
(\partial_{\omega_0}u_{\pm})(x,k\w) 
= 
e^{ik\jap{x,\w}}
[ik\jap{\w_{0},\w}e^{i\phi_{\pm}(x,\w)} 
+ 
i\jap{\w_{0},\w}\wt{w}_{\pm}(x,k\w) 
+ 
(\partial_{\omega_0}w_{\pm})(x,k\w)].
\end{equation*}

Now some elementary algebra shows that formula \eqref{eq.b15} can be rewritten as
\begin{equation}
\label{eq.b21}
s_{0}^{(N)}(\w,\w';k) 
= 
\left(\frac{k}{2\pi}\right)^{d-1}
\int_{\Lambda_{\w_{0}}}e^{-ik\jap{\w-\w',x}}a(\w,\w',x)\dd x,
\end{equation}
where
\begin{equation}
\label{eq.b22}
a(\w,\w',x) 
= 
\frac{1}{2}\jap{\w_{0},\w+\w'}\exp\bigl(i\phi_{-}(x,\w') - i\phi_{+}(x,\w)\bigr) + k^{-1}a_{1}(\w,\w',x),
\end{equation}
\begin{multline}
\label{eq.b23}
2a_{1}(\w,\w',x) 
=
\jap{\w,\w_{0}}
\bigl(e^{-i\phi_+(x,\omega)}\wt w_-(x,k\omega')+\overline{\wt w_+(x,k\omega)}w_-(x,k\omega')\bigr)
\\
+
\jap{\w',\w_{0}}
\bigl(e^{i\phi_-(x,\omega')}\overline{\wt w_+(x,k\omega)}+\overline{w_+(x,k\omega)}\wt w_-(x,k\omega')\bigr)
\\
+ 
i\overline{(\partial_{\omega_0}w_{+})(x,k\w)} w_{-}(x,k\w')- i\overline{w_{+}(x,k\w)}(\partial_{\omega_0}w_{-})(x,k\w')- 
2\jap{A(x),\w_{0}}\overline{w_{+}(x,k\w)}w_{-}(x,k\w').
\end{multline}

Note that the choice 
$a(\w,\w',x) = \frac{1}{2}\jap{\w+\w',\w_{0}}$ in \eqref{eq.b21} yields a $\delta$-function on the sphere. 
Thus, we can write
\begin{equation}
\label{eq.b24}
s_{0}^{(N)}(\w,\w';k) -\delta(\w-\w') 
= 
\left(\frac{k}{2\pi}\right)^{d-1}
\int_{\Lambda_{\w_{0}}}e^{-ik\jap{\w-\w',x}}(a_{0} + k^{-1}a_{1})(\w,\w',x)\dd x,
\end{equation}
where 
\begin{equation}
\label{eq.b25}
a_{0}(\w,\w',x) 
= 
\frac{1}{2}\jap{\w_{0},\w+\w'}\bigl(\exp(i\phi_{-}(x,\w') - i\phi_{+}(x,\w))-1\bigr).
\end{equation}

3) Let us change variables in the integral \eqref{eq.b24}. Instead of integrating over $x\in\Lambda_{\w_{0}}$, we shall integrate over $\xi\in\Lambda_{\w}$, where
\begin{equation}
\label{eq.b26}
x = \xi - \frac{\jap{\xi,\w_{0}}}{\jap{\w+\w',\w_{0}}}(\w+\w').
\end{equation}
Recall that $\omega,\omega'\in\Omega(\omega_0)$, and so the denominator
in \eqref{eq.b26} does not vanish.
An inspection shows that
\begin{equation*}
\jap{x,\w-\w'} = \jap{\xi,\w-\w'}.
\end{equation*}
Thus, we obtain
\begin{equation}
\label{eq.b27}
s_{0}^{(N)}(\w,\w';k) - \delta(\w-\w') 
= 
\left(\frac{k}{2\pi}\right)^{d-1}
\int_{\Lambda_{\w}}e^{-ik\jap{\w-\w',\xi}}b(\w,\w',\xi)\dd\xi,
\end{equation}
where $b=b_{0}+k^{-1}b_{1}$ with
\begin{equation}
\label{eq.b28}
b_{j}(\w,\w',\xi) = J(\w,\w')a_{j}(\w,\w',x(\xi))
\end{equation}
and $J(\w,\w')$ denotes the Jacobian of the linear map \eqref{eq.b26} considered as a map from $\Lambda_{\w}$ to $\Lambda_{\w_{0}}$. It is easy to see that $J(\w,\w')$ is a smooth function of $\w,\w'\in\Omega(\omega_0)$ and
\begin{equation}
\label{eq.b29}
J(\w,\w) = \jap{\w,\w_{0}}^{-1}.
\end{equation}
4) 
The r.h.s. of \eqref{eq.b27} is a semiclassical $\Psi$DO with the amplitude $b$
and a semiclassical parameter $h=k^{-1}$, 
see \eqref{eq.b8}. 
In order to complete the proof, by Proposition~\ref{prop.b3} if suffices to check the estimates
\begin{equation}
\label{eq.b30}
\vert\partial_{x}^{\alpha}\partial_{\w}^{\beta}\partial_{\w'}^{\gamma}b_{j}(\w,\w',\xi)\vert\leq C_{\alpha\beta\gamma}\jap{\xi}^{-m-\vert\alpha\vert},
\end{equation}
for $j=0,1$ and all multi-indices $\alpha,\beta,\gamma$ uniformly over $k\geq 1$, and to check the identity
\begin{equation}
\label{eq.b31}
b_{0}(\w,\w,\xi) = \exp(iM(\w,\xi)) - 1.
\end{equation}
Let us first check \eqref{eq.b31}. Recalling the definition \eqref{eq.a4}
of $M$ and the definition \eqref{eq.b10} of $\phi_{\pm}$, we get
\begin{equation*}
M(\w,\xi) = \phi_{-}(\xi,\w) - \phi_{+}(\xi,\w).
\end{equation*}
From this and \eqref{eq.b25}, \eqref{eq.b28} and \eqref{eq.b29}, we obtain
\begin{equation*}
b_{0}(\w,\w,\xi) = J(\w,\w)a_{0}(\w,\w,x(\xi)) = \exp\bigl(iM(\w,x(\xi))\bigr) - 1,
\end{equation*}
where $x(\xi)$ is the linear map \eqref{eq.b26}.
Next, by the definition of the map  $x(\xi)$, for $\omega=\omega'$ 
it takes the form 
$x(\xi) =  \xi + c\w$, and by the definition of the function $M$ we have
\begin{equation*}
M(\w,\xi+c\w) = M(\w,\xi).
\end{equation*}
Thus, we obtain \eqref{eq.b31}.\\
5) It remains to check that the estimates \eqref{eq.b30} are satisfied. This is essentially a consequence of Proposition~\ref{prop.b4}; let us check this. 
Recalling that $m=\min\{1,\rho-1\}$ and using the estimates \eqref{eq.b12}, \eqref{eq.b13}, we obtain
\begin{align*}
\vert\partial_{x}^{\alpha}\partial_{\w}^{\beta}\phi_{\pm}(x,\w)\vert &\leq C_{\alpha\beta}\jap{x}^{-m-\vert\alpha\vert},\\
\vert\partial_{x}^{\alpha}\partial_{\w}^{\beta}\wt{w}_{\pm}(x,k\w)\vert &\leq C_{\alpha\beta}\jap{x}^{-m-\vert\alpha\vert},\\
\vert\partial_{x}^{\alpha}\partial_{\w}^{\beta}(\partial_{\omega_0}w_{\pm})(x,k\w)\vert &\leq C_{\alpha\beta}\jap{x}^{-m-\vert\alpha\vert},\\
\vert\partial_{x}^{\alpha}\partial_{\w}^{\beta}w_{\pm}(x,k\w)\vert &\leq C_{\alpha\beta}\jap{x}^{-\vert\alpha\vert},
\end{align*}
where all the estimates are uniform in $k\geq 1$. It follows that $a_{0}$ and $a_{1}$, defined by \eqref{eq.b25}, \eqref{eq.b23} respectively, satisfy
\begin{equation}
\label{eq.b32}
\vert\partial_{x}^{\alpha}\partial_{\w}^{\beta}\partial_{\w'}^{\gamma}a_{j}(\w,\w',x)\vert \leq C_{\alpha\beta}\jap{x}^{-m-\vert\alpha\vert}
\end{equation}
uniformly in $k\geq 1$. 
Now from \eqref{eq.b26}, \eqref{eq.b28} and  \eqref{eq.b32} by an elementary calculation we obtain \eqref{eq.b30}.
\end{proof}
\section{Proof of Theorem~\ref{thm.a1}}
\label{sec.c}
\subsection{The case of a monomial \protect$\varphi$}
\label{sec.c1}
\begin{lemma}
\label{prop.c1}
Assume the hypothesis of Theorem~\ref{thm.a1}.
Then for any integers $\ell_1\geq0$, $\ell_2\geq0$ such that 
$\ell_1+\ell_2>(d-1)/m$, $m=\min\{1,\rho-1\}$, the asymptotic formula 
\eqref{eq.c1} holds true. 
\end{lemma}
\begin{proof}
By Lemma~\ref{prop.b6}, we have 
\begin{equation}
(S(k)-I)^{\ell_1}(S(k)^*-I)^{\ell_2}
=
(\Op_{k^{-1}}[\sigma]+R_q(k))^{\ell_1}
\bigl((\Op_{k^{-1}}[\sigma])^*+R_q(k)^*\bigr)^{\ell_2},
\label{eq.c1a}
\end{equation}
where $\sigma$, $R_q(k)$ are as described in Lemma~\ref{prop.b6}.
Expanding the brackets in \eqref{eq.c1a}, we obtain  
\begin{equation}
(S(k)-I)^{\ell_1}(S(k)^*-I)^{\ell_2}
=
(\Op_{k^{-1}}[\sigma])^{\ell_1}
\bigl((\Op_{k^{-1}}[\sigma])^*\bigr)^{\ell_2}
+Q_q(k),
\label{eq.c4}
\end{equation}
where $Q_q(k)$ is the sum of the products of operators, to be estimated below. 
For the first term in the r.h.s in \eqref{eq.c4}, 
by Proposition~\ref{prop.b2}, we have
\begin{multline*}
\lim_{k\to\infty}\left(\frac{k}{2\pi}\right)^{-d+1}
\Tr\bigl((\Op_{k^{-1}}[\sigma])^{\ell_1} \bigl((\Op_{k^{-1}}[\sigma])^*\bigr)^{\ell_2}\bigr)
\\
= 
\int_{\sph}\int_{\Lambda_{\w}}
\bigl(e^{iM(\w,\xi)}-1\bigr)^{\ell_1}\bigl(e^{-iM(\w,\xi)}-1\bigr)^{\ell_2}
\dd\xi\dd\w.
\end{multline*}
Let us check that by a suitable choice of $q$ we can ensure that the error term 
$\Tr Q_q(k)$ remains bounded as $k\to\infty$; this will certainly yield 
the desired asymptotics \eqref{eq.c1}. 
By choosing $q$ sufficiently large, we can make sure that the estimate
($\norm{\cdot}_1$ is the trace norm)
\begin{equation}
\norm{R_q(k)}_1=O(1), 
\quad k\to\infty,
\label{eq.c1b}
\end{equation}
holds true.
Next, using the estimates
$$
\abs{\Tr(AB)}
\leq
\norm{A}\norm{B}_1
\text{ and }
\norm{C}\leq \norm{C}_1
$$
and recalling that $Q_q(k)$ arose as a remainder term in the expansion 
of the brackets in the l.h.s. of \eqref{eq.c1a}, we obtain  
\begin{equation}
\label{eq.c3}
\abs{\Tr(Q_q(k))}
\leq 
C(\ell_1,\ell_2)
\max_{1\leq j\leq \ell_1+\ell_2}
\{
\norm{R_q(k)}^{j}_{1}, \Norm{\Op_{k^{-1}}[\sigma]}^{\ell_1+\ell_2-j}\}.
\end{equation}
By \eqref{eq.b3b}, we have
$$
\norm{\Op_{k^{-1}}[\sigma]}=O(1),
\quad k\to\infty.
$$
Combining the last inequality with \eqref{eq.c1b}, we obtain
that $\Tr(Q_q(k))$ is bounded as $k\to\infty$,
as required. 
\end{proof}


\subsection{Application of the Weierstrass approximation theorem}
\label{sec.c2}
\begin{lemma}
\label{lem.c3}
Let $\ell_0$ be an even natural number and let $\nu$, $\nu_k$ ($k\geq1$) 
be $\sigma$-finite measures on $\T\setminus\{1\}$ such that 
\begin{equation}
\int_\T \abs{z-1}^{\ell_0}\dd \nu(z)<\infty,
\quad
\int_\T \abs{z-1}^{\ell_0}\dd \nu_k(z)<\infty,
\label{eq.c14a}
\end{equation}
for all $k$. 
Suppose that for all integers $\ell_1\geq0$, $\ell_2\geq0$ 
such that $\ell_1+\ell_2\geq \ell_0$, the relation 
\begin{equation}
\lim_{k\to\infty}
\int_\T (z-1)^{\ell_1}(\overline{z}-1)^{\ell_2}d\nu_k(z)
=
\int_\T (z-1)^{\ell_1}(\overline{z}-1)^{\ell_2}d\nu(z)
\label{eq.c14}
\end{equation}
holds true. 
Then for any $\varphi\in C(\mathbb{T})$ such that $\varphi(z)\abs{z-1}^{-\ell_0}$ 
is continuous on $\T$, the relation 
\begin{equation}
\label{eq.c15}
\lim_{k\to\infty}
\int_{\mathbb{T}}\varphi(z)\dd\nu_{k}(z) 
= 
\int_{\mathbb{T}}\varphi(z)\dd\nu(z).
\end{equation}
holds true. 
\end{lemma}
\begin{proof}
Applying the Weierstrass approximation theorem to the  function $\varphi(z)\abs{z-1}^{-\ell_0}$, 
for any $\e>0$ we obtain a polynomial $\varphi_0(z)$ in $z$, $\overline{z}$ such that
\begin{equation*}
\vert\varphi(z)\vert z-1\vert^{-\ell_{0}}-\varphi_{0}(z)\vert\leq\e,\quad\forall z\in \mathbb{T}.
\end{equation*}
Let us define $\varphi_{\pm}(z)=(\Re\varphi_{0}(z)\pm\e)\vert z-1\vert^{\ell_{0}}$, then it follows from the above that
\begin{equation}
\label{eq.c16}
\varphi_{-}(z)\leq\Re\varphi(z)\leq\varphi_{+}(z)\quad\forall z\in \mathbb{T},
\end{equation}
\begin{equation}
\label{eq.c17}
\varphi_{+}(z)-\varphi_{-}(z)=2\e\vert z-1\vert^{\ell_{0}}.
\end{equation}
By the construction of $\varphi_\pm$, 
it can be represented as a polynomial in $w=z-1$, $\overline{w}=\overline{z}-1$
involving only products $w^{\ell_1}\overline{w}^{\ell_2}$
with $\ell_1+\ell_2\geq \ell_0$. 
Thus, by \eqref{eq.c16}, \eqref{eq.c17} we can write
\begin{gather}
\int_{\mathbb{T}}\varphi_-(z)\dd\nu(z)
\leq
\int_{\mathbb{T}}\Re\varphi(z)\dd\nu(z)
\leq
\int_{\mathbb{T}}\varphi_+(z)\dd\nu(z),
\label{eq.c17a}
\\
\int_{\mathbb{T}}\varphi_-(z)\dd\nu_k(z)
\leq
\int_{\mathbb{T}}\Re\varphi(z)\dd\nu_k(z)
\leq
\int_{\mathbb{T}}\varphi_+(z)\dd\nu_k(z),
\label{eq.c17c}
\\
\int_{\mathbb{T}}\varphi_+(z)\dd\nu(z)
-
\int_{\mathbb{T}}\varphi_-(z)\dd\nu(z)
=
2\e\int_{\T}\abs{z-1}^{\ell_0}d\nu(z),
\label{eq.c17b}
\end{gather}
where all integrals are absolutely convergent by \eqref{eq.c14a}.
Now we can use \eqref{eq.c14} to pass to the limit in \eqref{eq.c17c}.
Using \eqref{eq.c17a}, \eqref{eq.c17b} and denoting 
by $C$ the value of the integral in the r.h.s. of \eqref{eq.c17b}, we obtain
\begin{gather*}
\limsup_{k\to\infty}\int_{\mathbb{T}}\Re\varphi(z)\dd\nu_{k}(z)
\leq
\int_{\mathbb{T}}\varphi_{+}(z)\dd\nu(z)
\leq
\int_{\mathbb{T}}\Re\varphi(z)\dd\nu(z)+2\e C,
\\
\liminf_{k\to\infty}\int_{\mathbb{T}}\Re\varphi(z)\dd\nu_{k}(z)
\geq
\int_{\mathbb{T}}\varphi_{-}(z)\dd\nu(z)
\geq
\int_{\mathbb{T}}\Re\varphi(z)\dd\nu(z)-2\e C.
\end{gather*}
Since $\e>0$ may be taken arbitrary small, this yields
$$
\lim_{k\to\infty}
\int_{\T}
\Re \varphi(z)\dd\nu_k(z)
=
\int_{\T}\Re \varphi(z)\dd\nu(z).
$$
Since the same argument can be applied to the imaginary part of $\varphi$, 
we obtain the required statement. 
\end{proof}

\subsection{Proof of Theorem~\ref{thm.a1}}
Recall that the measures $\mu_k$ and $\mu$ are defined in \eqref{eq.muk} and \eqref{eq.a6}.
By \eqref{eq.a7a}, \eqref{eq.a7b}, 
the conclusion of Lemma~\ref{prop.c1} can be written as 
$$
\lim_{k\to\infty}
k^{-d+1}
\int_\T (z-1)^{\ell_1}(\overline{z}-1)^{\ell_2}d\mu_k(z)
=
\int_\T (z-1)^{\ell_1}(\overline{z}-1)^{\ell_2}d\mu(z).
$$
Now it remains to apply Lemma~\ref{lem.c3} with 
$\nu_k=k^{-d+1}\mu_k$ and $\nu=\mu$. 

\section*{Acknowledgements}
The authors are grateful to N.~Filonov and D.~Yafaev for making valuable 
remarks on the text of the preliminary version of the  paper.


\end{document}